\documentclass[a4paper]{article}

\usepackage{amsfonts}
\usepackage{amsmath}
\usepackage{amssymb}
\usepackage{amstext}
\usepackage{amsthm}
\usepackage{mathrsfs}

\newcommand{\C}{\mathbb C}
\newcommand{\R}{\mathbb R}

\newcommand{\eps}{\varepsilon}

\newtheorem{theorem}{Theorem}[section]
\newtheorem{defin}{Definition}[section]
\newtheorem{lemma}{Lemma}[section]
\newtheorem{prop}{Proposition}[section]

\DeclareMathOperator{\diam}{diam}

\DeclareMathOperator{\Id}{Id}
\DeclareMathOperator{\cat}{cat}

\begin{document}
\title{Semiclassical limit for the nonlinear Klein Gordon equation
in bounded domains}
\author{{ Marco G. Ghimenti}\\
{ Carlo R. Grisanti}\\
{\it\small Dipartimento di Matematica Applicata},\\ {\it\small Universit\`a di Pisa},
{\it\small via Buonarroti 1c, 56100, Pisa, Italy}\\
{\it\small e-mail: ghimenti@mail.dm.unipi.it}\\
{\it\small grisanti@dma.unipi.it}
}
\date{}
\maketitle
\begin{center}
{\bf\small Abstract}

\vspace{3mm}
\hspace{.05in}\parbox{4.5in}
{{\small We are interested to the existence of standing waves for the nonlinear
Klein Gordon equation
$\eps^2\Box\psi+W'(\psi)=0$ in a bounded domain $D$.
A standing wave has the
form $\psi(t,x)=u(x)e^{-i\omega t/\eps}$;
for these solutions the Klein Gordon equation becomes
\begin{equation}\label{abseq}\tag{$\dagger$}
\left\{
\begin{array}{ll}
-\eps^2\Delta u+W'(u)=\omega^2u& x\in D\\
(u,\omega)\in H^1_0(D)\times\R
\end{array}
\right.
\end{equation}
and we want to use a Benci-Cerami type argument in order to prove a
the existence of several standing waves localized in suitable points of $D$.

The main result of this paper is that, under suitable growth condition on $W$, for $\eps$
sufficiently small, we have at least $\cat(D)$ stationary solution of equation (\ref{abseq}),
while $\cat(D)$ is the Ljusternik-Schnirelmann category.

The proof is achieved by solving a constrained critical point problem via variational techniques.

}}
\end{center}


\section{Introduction}

We are interested to the stationary solutions of
\begin{equation}\label{kg}\tag{*}
\left\{
\begin{array}{ll}
\eps^2\Box\psi+W'(\psi)=0& (t,x)\in\R^+\times D\\
\psi(\cdot,x)\in H^1_0(D,\C)
\end{array}
\right.
\end{equation}
where $\displaystyle\Box=\left(\frac{\partial}{\partial t}\right)^2-\Delta x$
and $D$ is an open set in $\R^N$.

The nonlinear Klein--Gordon equation (\ref{kg}) is the simplest equation
invariant for the Poincar\'e group which admits solitary waves. By solitary wave
we mean a solution of a field equation whose energy travels as a localized packet;
solitary waves which exhibit orbital stability are
called solitons.
The Klein Gordon equation is the Euler--Lagrange equation of the functional
$$\int{\cal L}(\partial_t\psi,\nabla\psi,\psi)\,dx\,dt$$
where the Lagrangian density is given by
$${\cal L}=\frac{\eps^2}2|\partial_t\psi|^2-\frac{\eps^2}2|\nabla\psi|^2-W(\psi).$$
Since the equation is invariant with respect to the Poincar\'e
group, by the Noether theorem, it admits several integrals of
motion, which are preserved in time. In particular, we are
interested in the conservation of the energy and the charge, which
play a fundamental role in our framework. The energy $E_\varepsilon$
and the charge $C$ have the following expressions:
$$E_\varepsilon=\int\left(\frac{\eps^2}2|\partial_t\psi|^2+
\frac{\varepsilon^2}2|\nabla\psi|^2+W(\psi)\right)\,dx$$
$$C=\text{Im}\int\partial_t\psi\overline\psi\,dx$$

If $D=\R^N$, the study of solitons for equation (\ref{kg}) has a very long history starting with
the pioneering paper of Rosen \cite{Ros68}.
Coleman \cite{CGM78} and Strauss \cite{St78} gave the first
rigorous proofs of existence of solutions of the type
(\ref{standing}) for particular forms of $W'$, and later necessary and
sufficient existence conditions have been found by Berestycki and Lions \cite{BL83}.

The first orbital stability result for (\ref{kg}) is due to
Shatah; in \cite{Sh83} a necessary and sufficient condition for
orbital stability is given. See also \cite{GSS87},
 for a generalization of the methods used in \cite{Sh83}.

Recently, in \cite{BBBM07,BBGM07},
the role played by the energy/charge ratio has been exploited in the existence
of solitons. In particular, if the previous ratio is small enough, we can find
solitons. The properties of solitons and a general approach to field equations
using the energy/charge ratio are studied in \cite{BBBS}; in particulare the
approach presented in \cite{BBBS} is suitable also to study the existence of vortices,
i.e. solitary waves with non-vanishing angular momentum. See also \cite{BBR,BB} for the
existence of vortices in the wave equation.

In this paper we are interested in equation (\ref{kg}) when
$D$ is bounded. In this case it is possible to prove the
existence of standing waves localized in suitable points of $D$.

A standing wave is a particular type of solitary wave having the
following form
\begin{equation}\label{standing}
\psi(t,x)=u(x)e^{-i\omega t/\eps},\ u\ge0,\ \omega\in\R.
\end{equation}
For solutions of the type (\ref{standing}) the Klein Gordon equation (\ref{kg}) becomes
\begin{equation}\label{maineq}\tag{$\dagger$}
\left\{
\begin{array}{ll}
-\eps^2\Delta u+W'(u)=\omega^2u& x\in D\\
(u,\omega)\in H^1_0(D)\times\R
\end{array}
\right.
\end{equation}
and we want to use a Benci-Cerami type argument \cite{BC94} in order to prove a
multiplicity result.

The energy and the charge for standing waves solutions take the form:
\begin{equation}
E_\eps(u,\omega)=\int\limits_D
\eps^2\frac{|\nabla u|^2}{2}+W(u)+\frac{\omega^2u^2}{2} dx;\quad
\displaystyle C=\omega\int|u|^2
\end{equation}
We know (see \cite{BBBM07}) that $\psi(t,x)=u_0(x)e^{\frac{-i\omega_0 t}{\eps}}$
is a stationary solution of (\ref{kg})
if and only if  $(u_0,\omega_0)$ is a critical point of the energy functional $E_\varepsilon$
constrained on the manifold
\begin{equation}
M_C(D)=
\left\{(u,\omega)\in H^1_0(D)\times\R\ :\ \omega\int\limits_D u^2dx =
C\right\},
\end{equation}

We made the
following assumptions on the nonlinearity $W$:
$\displaystyle W(s)=\frac{\Omega^2}{2}|s|^2+N(s)$ where $\Omega\in\R$
and $N(s)$ is a radially symmetric function with $N(0)=N'(0)=N''(0)=0$ which
satisfies
\begin{equation}\label{Np}\tag{$N_p$}
|N'(s)|\leq c_1|s|^{q-1}+c_2|s|^{p-1}
\end{equation}
where $c_1$ and $c_2$ are positive constants and $2<q\leq p<2^*$.
We need (\ref{Np}) in order to have a $C^1$ energy functional.

We also require the following hypotheses
\begin{equation}\label{W1}\tag{$W_1$}
W(s)\geq 0;
\end{equation}
\begin{equation}\label{W2}\tag{$W_2$}
\exists s_0 \text{ s.t. }N(s_0)<0.
\end{equation}

The hypothesis (\ref{W2}) is quite general and seems to be
necessary in order to have the energy/charge ratio which ensures
stable solitary waves.

The main result of this paper is the following:
\begin{theorem}\label{mainteo}
Let $\sigma$ be a charge sufficiently large. If $W$ satisfy (\ref{Np}),(\ref{W1}) and
(\ref{W2}) then, for $\eps$
sufficiently small, we have at least $\cat(D)$ stationary solution of
equation (\ref{maineq}) on the manifold $M_{\sigma\eps^N}(D)$.
\end{theorem}

\section{Useful estimates}
In the next, we often denote simply $E(u)$ for $E_1(u)$.

\

\begin{defin}
Fixed $\sigma$, we set, for any $\eps \in\R^+$ and for any $D\subseteq\R^N$
\begin{equation}
m_\sigma(\eps,D):=\inf\{E_\eps(u,\omega)\ :\ (u,\omega)\in M_{\sigma\eps^N}(D)\}
\end{equation}
\end{defin}
We know, by \cite{BBBM07} that
for a sufficiently large $\sigma$ the infimum
$m(1,\R^N)$ is attained by a positive radially symmetric function.
Indeed, if $(u,\omega)\in M_\sigma(\R^N)$ then
$\displaystyle u\left(\frac {x}{\eps}\right)\in M_{\sigma\eps^N}$ and
$\displaystyle E_\eps\left(u\left(\frac {x}{\eps}\right)\right)=\eps^NE(u)$,
so for
$\sigma$ sufficiently large we have that
$m_\sigma(\eps,\R^N)$ is attained for all $\eps$ and it holds
\begin{equation*}
m_\sigma(\eps,\R^N)=\eps^Nm_\sigma(1,\R^N)
\end{equation*}
Furthermore, if $D$ is a compact set, we have the following result.
\begin{prop}
For all compact set $D$, for all $\eps>0$, and for every charge $\sigma$,
$m_\sigma(\eps,D)$ is attained.
\end{prop}
\begin{proof}
At first, we notice that, by (\ref{W1}), $m_\sigma(\eps,D)>0$.
Now, let $(u_n,\omega_n)$ a minimizing sequence in $M_{\sigma\eps^N}(D)$. We have that
$u_n$ is bounded in $H^1$ and that $\omega_n$ is bounded in $\R$. The boundedness
of $(u_n,\omega_n)$ is proven in \cite{BBBM07}, and we report here only the main lines.

We can write the energy as
\begin{equation}
E_\eps(u_n,\omega_n)= \int \eps^2\frac{|\nabla u_n|^2}{2}+W(u_n)\ dx
\ +\frac{|\sigma||\omega_n|}{2}.
\end{equation}
So, we have that $\omega_n$ and $\displaystyle\int |\nabla u_n|^2 dx$ are bounded.

Now, because $W(0)=W'(0)=0$ and $W''(0)=\Omega^2$, we can stand that
\begin{equation}\label{stimclaudio}
\exists \delta>0 \ \exists \beta_1>0, \text{ such that } W(s) \geq
\beta_1 s^2 \text { for } 0\leq |s| \leq \delta.
\end{equation}

We show that $\displaystyle\int u_n^2dx$ is bounded. Let us suppose, by contradiction, that
$\displaystyle\int u_n^2dx \rightarrow \infty$. We have that
$\displaystyle\int W(u_n)dx$ is bounded, so, by (\ref{stimclaudio}),
\begin{equation}\label{for}
\int W(u_n)dx \geq \int_{0 \leq u_n\leq \delta} W(u_n)dx \geq
\beta_1\int_{0 \leq u_n \leq \delta} u_n^2 dx .
\end{equation}
On the other hand
\begin{displaymath}
\int_{0\leq u_n\leq \delta} u_n^2 dx +\int_{u_n\geq \delta}
u_n^2dx \rightarrow \infty,
\end{displaymath}
thus we have, by equation (\ref{for}), that
$\displaystyle
\int_{u_n\geq \delta} u_n^2dx \rightarrow \infty$.
This drives to a contradiction because
\begin{equation}
\frac{1}{\delta^{2^*-2}} \ \int_{u_n\geq \delta} u_n^{2^*}dx \geq
\int_{u_n\geq \delta} u_n^2dx
\end{equation}
and by the Sobolev theorem
\begin{equation}
\int_{u_n\geq \delta} u_n^{2^*}dx \leq \int u_n^{2^*}dx \leq K
\int |\nabla u_n|^2dx< \text{const}.
\end{equation}

Thus $u_n$ is bounded in $H^1$ and $\omega_n$ is bounded in $\R$.
Up to subsequences, we have that
\begin{eqnarray}
&&\omega_n\rightarrow \omega\in\R;\\
&&u_n\rightarrow u \text{ strongly in }L^p,\ 2\leq p<2^*;\\
&&u_n\rightharpoonup u \text{ weakly in }H^1.
\end{eqnarray}
We have that $(u,\omega)\in M_{\sigma\eps^N}(D)$, and, by (\ref{Np}),
$W(u_n)\rightarrow W(u)$. Finally,
\begin{equation}
m_\sigma(\eps,D)=\liminf_{n\rightarrow\infty}E_\eps(u_n,\omega_n)\geq E_\eps(u,\omega)
\end{equation}
that conlcudes the proof.
\end{proof}
At last, by \cite{GNN79},
if $D=B(0,\rho)$ we have that $u$ is
positive, radially symmetric and satisfies the ordinary differential
equation
\begin{equation}\label{ode}
-\eps^2\frac{d^2u}{dr^2}-\frac{\eps^2(N-1)}{r}\frac{du}{dr}+W'(u)=\omega^2u
\end{equation}

From now on, we fix $\sigma$ sufficiently large in order to have that
$m_\sigma(\eps,\R^N)$ is attained. We remark that, if we require a stronger version of assumption
(\ref{W2}), namely
\begin{equation}\label{W2b}\tag{$W_2'$}
N(s)\leq -|s|^{2+\eps}, \text{ with }0<\eps<\frac 4N,\text{ for }|s|\text{ small,}
\end{equation}
we have that $m_\sigma(\eps,\R^N)$ is attained for all $\sigma$. In this case all our result can
be extended in a trivial way. For the proof of the existence of the minimizer for
every $\sigma$ and for the discussion on the hypothesis (\ref{W2b}) we refer to
\cite{BBBM07,BBGM07}.
\medskip

Hereafter, since we fix $\sigma$, we note simply $m_\sigma(\eps,D)$
by $m(\eps,D)$.

\begin{defin}
We set
\begin{equation}
m(\eps,\rho):=m(\eps,B_\rho(y)).
\end{equation}
\end{defin}
We notice that
$m(\eps,\rho)$ is well defined because its value does not depend on $y$.

\begin{lemma}\label{rho1rho2}
If $\rho_1<\rho_2$, then $m(\eps,\rho_1)>m(\eps,\rho_2)$
\end{lemma}
\begin{proof}
Let $(\bar u,\bar \omega)\in M_{\sigma\eps^N}(B_{\rho_1})$ be such that
\begin{equation}
E_\eps(\bar u, \bar \omega)=m(\eps,\rho_1).
\end{equation}
We can define $\tilde u\in B_{\rho_2}$ by
\begin{equation}
\tilde u=\left\{
\begin{array}{cl}
u&\text{in } B_{\rho_1};\\
0&\text{outside}.
\end{array}
\right.
\end{equation}
Thus we have $m(\eps,\rho_1)\geq m(\eps,\rho_2)$. Suppose that the
equality holds. Then $\tilde u$ satisfies the ordinary
differential equation (\ref{ode}), but in this case $\tilde u\equiv 0$, and
this is a contradiction.
\end{proof}

\begin{lemma}\label{rhoRn}
The relation
\begin{equation}
\lim_{\rho\rightarrow+\infty}m(1,\rho)=m(1,\R^N)
\end{equation}
holds.
\end{lemma}
\begin{proof}
Let $(\bar u, \bar \omega)\in M_\sigma(\R^N)$ such that
\begin{equation}
E(\bar u, \bar \omega)=m(1,\R^N);
\end{equation}
We know that $\bar u$ is radially symmetric. We choose a suitable cut off
$\chi_\rho\in C^\infty(\R^N)$ such that
\begin{eqnarray*}
&&\chi_\rho\equiv1\text{ if }|x|\leq\rho/2,\\
&&\chi_\rho\equiv0\text{ if }|x|\geq\rho,\\
&&|\nabla\chi_\rho|\leq2,
\end{eqnarray*}
and we define $w_\rho=\chi_\rho\bar u$. There exists a $t_\rho>0$ such that,
$\displaystyle \bar \omega\int |t_\rho w_\rho|^2=\sigma$.
We set $u_\rho=t_\rho w_\rho$,
so $u_\rho\in M_{\sigma}(B_\rho)$. Thus, by Lemma \ref{rho1rho2},
\begin{equation}
E(u_\rho, \bar \omega)\geq m(1,\rho)>m(1,\R^N).
\end{equation}
We want to prove that
\begin{equation}
\lim_{\rho\rightarrow\infty}E(u_\rho,\bar \omega)=m(1,\R^N).
\end{equation}
We have that $w_\rho\rightarrow \bar u$ in $L^2(\R^N)$ when
$\rho\rightarrow\infty$,
so $t_\rho\rightarrow1$ and $u_\rho\rightarrow \bar u$ in $L^2(\R^N)$.
At last we have
\begin{eqnarray*}
\int\limits_{\R^N}|\nabla w_\rho-\nabla \bar u|^2&=&
\int\limits_{\R^N}|(\nabla \chi_\rho)\bar u+(1-\chi_\rho)\nabla\bar u|^2\leq\\
&\leq&2\int\limits_{\R^N}|(\nabla \chi_\rho)\bar u|^2+2\int\limits_{\R^N}|(1-\chi_\rho)\nabla\bar u|^2\leq\\
&\leq&4\int\limits_{B_{\rho/2}^C}|\bar u|^2++2\int\limits_{B_{\rho/2}^C}|\nabla\bar u|^2\rightarrow 0
\end{eqnarray*}
when $\rho\rightarrow\infty$. So $\nabla w_\rho\rightarrow\nabla\bar u $ in $L^2$ and
$\nabla u_\rho\rightarrow\nabla\bar u $ in $L^2$.
Thus $u_\rho\rightarrow \bar u$ in $H^1(\R^N)$ and
\begin{equation}
\lim_{\rho\rightarrow\infty}E(u_\rho,\bar \omega)=E(\bar u,\bar \omega)=m(1,\R^N),
\end{equation}
that concludes the proof.
\end{proof}
\begin{defin}
For any $u$ in $H^1(\R^N)$ with compact support we define the barycentre map
\begin{equation}
\beta(u):=\frac{\displaystyle\int_{\R^N} x\cdot |\nabla u(x)|^2 dx}
{\displaystyle\int_{\R^N}  |\nabla u(x)|^2 dx}
\end{equation}
\end{defin}

\begin{defin}
We define, for every $\rho>0$ and for every $\gamma>1$,
\begin{equation*}
m^*(\eps,\rho,\gamma):=
\inf\{E_\eps(u,\omega)\ :\ (u,\omega)\in
M_{\sigma\eps^N}(B_{\gamma\rho}(0)\smallsetminus B_{\rho}(0)),\beta(u)=0\}.
\end{equation*}
\end{defin}

We notice that moving the center of the ball and $\beta(u)$ does not affect
$m^*(\eps,\rho,\gamma)$. Also, we remark that
\begin{equation}
m^*(\eps,\rho,\gamma)\geq m(\eps,\gamma\rho)>m(\eps,\R^N).
\end{equation}

\begin{defin}
We define
\begin{equation}
m^*(\eps,\gamma)=\inf_{\rho>0}m^*(\eps,\rho,\gamma).
\end{equation}
\end{defin}

\begin{lemma}\label{gammaRn}
The inequality
\begin{equation}
m^*(1,\gamma)>m(1,\R^N)
\end{equation}
holds for any fixed $\gamma>1$.
\end{lemma}
\begin{proof}
The inequality $m^*(1,\gamma)\ge m(1,\R^N)$ follows trivially from set
inclusion. Let us suppose, by contradiction, that $m^*(1,\gamma)=
m(1,\R^N)$. In this case we can find a sequence of positive numbers
$(\rho_n)$ such that
$$\lim_{n\to\infty} m^*(1,\gamma,\rho_n)=m(1,\R^N).$$
We claim
that $\rho_n\to\infty$. Indeed, if there exists $L>0$ such that
$0<\rho_n<L$ then, as before
$$m^*(1,\rho_n,\gamma)\ge m(1,\gamma L)>m(1,\R^N).$$
Hence we can suppose that $(\rho_n)$ is an increasing unbounded
sequence. Next we extend the functions $u_n$ to zero outside
$B_{\gamma\rho_n}(0)\setminus B_{\rho_n}(0)$ obtaining a minimizing
sequence $(u_n,\omega_n)$ for the functional $E_1$ in $\R^N$.
From step1 of Theorem 5 in ?? we have that $u_n$ converges to
$\overline u\not=0$ in $L^t_{\text{loc}}(\R^N)$ with $2\le t\le2^*$ and we
get a contradiction since the support of $u_n$ is contained in
$\R^N\setminus B_{\rho_n}(0)$ and $\rho_n\to\infty$.
\end{proof}

\begin{lemma}\label{Rgamma}
For any $\gamma>1$ there exists $\bar R=\bar R(\gamma)$ such that, for any
$R>\bar R$ we have
\begin{equation}
m(1,R)<m^*(1,R,\gamma)
\end{equation}
\end{lemma}
\begin{proof}
Let us fix $\gamma>1$. It follows straightforward from the definition that
$$m^*(1,\gamma)\le m^*(1,\rho,\gamma)\qquad\forall \rho>0$$
From Lemma \ref{gammaRn} we have that $m^*(1,\gamma)>m(1,\R^N)$ and, by Lemma
\ref{rhoRn}, there exists $\overline R(\gamma)$ such that
$$m^*(1,\gamma)>m(1,R)\qquad\forall R>\overline R(\gamma).$$
This completes the proof.
\end{proof}

\begin{lemma}\label{mmstar}
For every $\rho>0$ and $\gamma>1$ there exists
$\bar\eps=\bar\eps(\rho,\gamma)$ such that, for any $0<\eps\leq \bar\eps$
\begin{equation}
m(\eps,\rho)<m^*(\eps,\rho,\gamma)
\end{equation}
\end{lemma}
\begin{proof}
Given a function $u\in H^1_0(\R^N)$ and a real number $\eps>0$ we set
$u_\eps(x)=u(\eps x)$. A simple change of variable shows that:
$$(u,\omega)\in M_\sigma(B_{\gamma\rho}(0)\setminus B_{\rho}(0))
\Longleftrightarrow
(u_\eps,\omega)\in M_{\sigma/\eps^N}(B_{\gamma\rho/\eps}(0)\setminus B_{\rho/\eps}(0)).$$
Furthermore we have that $\beta(u_\eps)=\frac1\eps\beta(u)$, and
$$E_1(u_\eps,\omega)=\int_{B_{\gamma\rho/\eps}\setminus
B_{\rho/\eps}}\frac{|\nabla u_\eps|^2}2+W(u_\eps)+\frac{\omega
u_\eps^2}2\,dx=$$
$$=\frac1{\eps^N}\int_{B_{\gamma\rho}\setminus B_\rho}\eps^2\frac{|\nabla
u|^2}2+W(u)+\frac{\omega u^2}2\,dx=\frac1{\eps^N} E_\eps(u,\omega).$$
Hence we get:
$$m^*(\eps,\rho,\gamma)=\eps^N
m^*\left(1,\frac\rho\eps,\gamma\right),\qquad
m(\eps,\rho)=\eps^N m\left(1,\frac\rho\eps\right).$$
Let us choose $\displaystyle\overline\eps<\frac\rho{\overline R(\gamma)}$ where $\overline
R(\gamma)$ is defined in Lemma \ref{Rgamma}. With such a choice we have that,
for every $\eps\in(0,\overline\eps)$
$$m(\eps,\rho)=\eps^N m\left(1,\frac\rho\eps\right)<\eps^N m^*\left(1,\frac\rho\eps,\gamma\right)=m^*(\eps,\rho,\gamma).$$
\end{proof}

\section{Main result}

We assume, without loss of generality, that $0\in D$. There exists an
$r>0$ such that the sets
\begin{eqnarray*}
D^+:=\{x\in\R^N\ :\ d(x,D)\leq r\}&\text{ and }&
D^-:=\{x\in\R^N\ :\ d(x,\partial D)\geq 2r\}
\end{eqnarray*}
are homotopically equivalent to $D$ and $B_r(0)\subset D$. We set
\begin{equation}
\gamma=\frac{2\diam D}{r}
\end{equation}

\begin{lemma}
There exists $\bar \eps$ such that, for all $0<\eps\leq \bar\eps$
\begin{equation}
(u,\sigma)\in M_{\sigma\eps^N}(D),\  E_\eps(u,\omega)<m(\eps,r)\Rightarrow\  \beta(u)\in D^+.
\end{equation}
\end{lemma}
\begin{proof}
Fixed $r,\gamma$ as above,
let $\eps<\bar\eps(r,\gamma)$ where $\bar \eps(r,\gamma)$ is defined
by Lemma \ref{mmstar}. Then $B_r(0)\subset D$ so $m(\eps,D)<m(\eps,r)$, so
the set
\begin{equation}
E_\eps^{m(\eps,r)}:=\{(u,\sigma)\in M_{\sigma\eps^N}(D),\  E_\eps(u,\omega)<m(\eps,r)\}
\end{equation}
is not empty.

Now, take $(u^*,\omega^*)\in E_\eps^{m(\eps,r)}$, and suppose that $x^*:=\beta(u^*)\not\in D^+$; because
$|\beta(u^*)|\leq \diam D$, we have that
\begin{displaymath}
D\subset B_{2\diam D}(x^*)\smallsetminus B_r(x^*)=
B_{\gamma r}(x^*)\smallsetminus B_r(x^*).
\end{displaymath}
Thus
\begin{eqnarray*}
m^*(\eps,r,\gamma)&=&
\inf\{E_\eps(u,\omega),\ (u,\omega)\in E_\eps^{m(\eps,r)},\ \beta(u)=x^*\}\leq\\
&\leq&E_\eps(u^*,\omega^*)<m(\eps,r)
\end{eqnarray*}
that contradicts Lemma \ref{mmstar}
\end{proof}

We want to define two continuous operator in order to prove the main theorem.
\begin{defin}
We define
\begin{eqnarray*}
&&B:H^1(D)\times\R\rightarrow\R^N;\\
&&B(u,\omega)=\beta(u).
\end{eqnarray*}
\end{defin}
The operator $B$ is well defined and continuous. Furthermore, if
$\eps<\bar \eps(r,\gamma)$ we have
\begin{equation*}
B(E_\eps^{m(\eps,r)})\subset D^+
\end{equation*}
as proved in the previous lemma.

Fixed $\sigma$ we can choose $(u_\eps,\omega_\eps)$ such that

\begin{eqnarray*}
&&E_\eps(u_\eps,\omega_\eps)= m(\eps, 2r)<m(\eps,r);\\
&&\omega_\eps=\frac{e^N\sigma}{\int|u_\eps|^2}.
\end{eqnarray*}
We know that $u_\eps$ is radially symmetric, so $\beta(u_\eps)=0$.
Of course we can extend  $u_\eps$ trivially by zero to a function
$\tilde u_\eps$ defined in $D$, in order to obtain a pair
$(\tilde u_\eps,\omega_\eps)\in M_{\eps^N\sigma}(D)$. With abuse of notation
in the next we will identify $\tilde u_\eps$ and $u_\eps$.
\begin{defin}
Fixed $\sigma$, for any $\eps$ we define
\begin{eqnarray*}
&&\Phi_\eps:D^-\rightarrow M_{\eps^N\sigma}(D);\\
&&\Phi_\eps(y)=(u_\eps(|x-y|),\omega_\eps).
\end{eqnarray*}
\end{defin}
It is easy to see that $\Phi_\eps$ is a continuous function.

\begin{proof}[Proof of Theorem \ref{mainteo}]
Fix $r$ such that $D^+$ and $D^-$ are homotopically equivalent to $D$.
Fix $\sigma$ sufficiently big in order to have that $m(1,\R^N)$ is attained,
and fix $\eps<\bar\eps(r,\gamma)$ as in Lemma \ref{mmstar}. We have that
\begin{equation}
\Phi_\eps(y)\subset E_\eps^{m(\eps,r)}
\end{equation}
for all $y\in D^-$. So, trivially
\begin{eqnarray*}
&&B(\Phi_\eps):D^-\rightarrow D^+;\\
&&B(\Phi_\eps)|_{D^-}\approx \Id_{D^-}.
\end{eqnarray*}
By a well known topological result, we have that
\begin{equation}
\cat (E_\eps^{m(\eps,r)} )\geq \cat(D),
\end{equation}
so we have at least $\cat(D)$ distinct stationary solution
of equation (\ref{kg}) with energy less that $m(\eps,r)$.
\end{proof}

\nocite{BC91,BC94}

\end{document}